\documentclass[letterpaper,12pt]{article}

\usepackage[letterpaper,top=3cm,bottom=3cm,left=3cm,right=3cm,marginparwidth=1.75cm]{geometry}

\usepackage[english]{babel}
\usepackage[utf8x]{inputenc}
\usepackage[T1]{fontenc}

\usepackage{amsmath}
\usepackage{amsthm}
\usepackage{amssymb}
\usepackage[all,cmtip]{xy}
\usepackage{graphicx}
\usepackage{theoremref}

\theoremstyle{plain}
\newtheorem{theorem}{Theorem}[section]
\newtheorem{lemma}[theorem]{Lemma}
\newtheorem{prop}[theorem]{Proposition}
\newtheorem{cor}[theorem]{Corollary}

\theoremstyle{definition}
\newtheorem{exx}[theorem]{Example}
\newtheorem{deff}[theorem]{Definition}
\newtheorem{remark}[theorem]{Remark}

\usepackage{theoremref}

\newcommand{\Z}{\mathbb{Z}}
\newcommand{\Q}{\mathbb{Q}}
\newcommand{\N}{\mathbb{N}}

\newcommand{\C}{\mathbb{C}}
\newcommand{\K}{\mathbb{K}}

\newcommand{\BAG}{B^A(G)}
\newcommand{\BAH}{B^A(H)}
\newcommand{\gBAG}{\widetilde{B^A}(G)}

\newcommand{\MAG}{\mathcal{M}^A_G}
\newcommand{\MAH}{\mathcal{M}^A_H}

\newcommand{\BnG}{B^{C_n}(G)}

\newcommand{\BnGz}{B^{C_n}_{\Z[\zeta]}(G)}
\newcommand{\gBnGz}{\widetilde{B^{C_n}_{\Z[\zeta]}}(G)}

\newcommand{\MnG}{\mathcal{M}^{C_n}_G}

\newcommand{\DnG}{\mathcal{D}^{C_n}_G}

\newcommand{\eGnGg}{e^{G,n}_{G,g}}

\newcommand{\sGnHh}{s^{G,n}_{H,h}}
\newcommand{\eGnHh}{e^{G,n}_{H,h}}
\newcommand{\cGnHh}{c^{G,n}_{H,h}}

\title{On fibered Burnside rings, fiber change maps and cyclic fiber groups}

\author{
Benjam\'in Garc\'ia\\
Centro de Ciencias Matem\'aticas\\
Universidad Nacional Aut\'onoma de M\'exico\\
Morelia, Michoac\'an 58089\\
M\'exico\\
\texttt{benjamingarcia@matmor.unam.mx}
\and 
Alberto G. Raggi-Cárdenas\\
Centro de Ciencias Matem\'aticas\\
Universidad Nacional Aut\'onoma de M\'exico\\
Morelia, Michoac\'an 58089\\
M\'exico\\
\texttt{agraggi@gmail.com}
}

\date{\today}

\begin{document}

\maketitle

\begin{abstract}
Fibered Burnside rings appear as Grothendieck rings of fibered permutation representations of a finite group, generalizing Burnside rings and monomial representation rings. Their species, primitive idempotents and their conductors are of particular interest in representation theory as they encode information related to the structure of the group. In this note, we introduce fiber change maps between fibered Burnside rings, and we present results on their functoriality and naturality with respect to biset operations. We present some advances on the conductors for cyclic fiber groups, and fully determine them in particular cases, covering a wide range of interesting examples.
\end{abstract}



\section*{Introduction}

The $A$-fibered Burnside ring of a finite group $G$, denoted by $\BAG$, is defined as the Grothendieck ring of the category of finite $A$-fibered permutation representations of $G$, where $A$ is an abelian group referred to as the fiber group. This ring was introduced by Dress in \cite{Dress} as a generalization of the Burnside ring (where $A$ is the trivial group) and the monomial representation ring (where $A=\C^{\times}$). It has been extensively studied as a commutative ring, as a Mackey functor, and as a global biset functor whose modules, known as fibered biset functors, are a source of important examples in the theory of biset functors (\cite{BolCos},\cite{Romero}). 

The ring homomorphisms from $\BAG$ to the field of complex numbers, known as the \textit{species} of $\BAG$, as well as their respective primitive idempotents in $B^A_{\K}(G):=\K\otimes \BAG$ for some splitting field $\K$, are particularly interesting since they are related to questions about solubility and nilpotency, among other aspects of the structure of the group \cite{Muller}. The species and idempotents were determined by Boltje \& Y{\i}lmaz in \cite{BolYil} under the assumption that the fiber group $A$ has finite $\text{exp}(G)$-torsion. More precisely, the cyclotomic field $\K=\Q[\zeta]$ with $\zeta=e^{\frac{2\pi}{n}i}$, $n=\text{exp}(A_{\text{exp}(G)})$, is a splitting field for $\BAG$; the primitive idempotents of $B^A_{\K}(G)$ are in bijection with the species of $\BAG$, which are in one-to-one correspondence with the orbits of the set $\mathcal{D}^A_G=\left\{(H,\Phi)\; \middle| \;H\leq G,\Phi \in\text{Hom}(\text{Hom}(H,A),\K^{\times})\right\}$ on which $G$ acts by conjugation. The species and idempotent corresponding to the class of a pair $(H,\Phi)$ are denoted by $s^{G,A}_{H,\Phi}$ and $e^{G,A}_{H,\Phi}$, respectively.

The conductor of the primitive idempotent $e^{G,A}_{H,\Phi}$ over $\Z[\zeta]$ is the least positive integer $t$ such that $te^{G,A}_{H,\Phi}$ lies in $B^A_{\Z[\zeta]}(G)$, that we denote by $c^{G,A}_{H,\Phi}$. These numbers are invariants of $G$ attached to the ring structure of $\BAG$, hence preserved under isomophisms between fibered rings; they are known for $B(G)$ (\cite{DressVall},\cite{KraThe}, \cite{Nicolson}) and $D(G)$ \cite{Muller}, from which we can recover information about $G$, as for instance, its order or if it is abelian. The induction, restriction and conjugation maps, as well as the natural inclusion of $B(G)$ in $D(G)$, play an important role in the computation of the conductors of $D(G)$. More generally, if $A\leq C$, the ring $B^A(G)$ is often identified as a subring of $B^C(G)$. In an even more general setting, any map of abelian groups $f:A\longrightarrow C$ induces a \textit{fiber change map} $f_*:\BAG \longrightarrow B^C(G)$ that happens to be a ring homomorphism preserving the standard basis.

The aim of this note is to present some results on fiber change maps and significant progress on the conductors of primitive idempotents for the case of a cyclic fiber group. We set up the background material in Sections 1 and 2. Section 1 is devoted to the study of the subgroups $G^{[n]}\unlhd G$ for $n\in \N$, defined as the minimal normal subgroup such that the quotient $G/G^{[n]}$ is an $n$-torsion group, as well as their connection to the $n$-torsion of the Pontryagin dual $\text{Hom}(G,\C^{\times})$. Section 2 is a summary of definitions and results on fibered Burnside rings and their standard bases, ghost rings, mark morphisms, and their biset functor structure.

In Section 3, we recall Boltje \& Y{\i}lmaz's formulae for species and idempotents, and we point out how they simplify for $A=C_n$, a cyclic group of order $n$. In this case, we take $\DnG$ to be the set of all pairs $(H,hH^{[n]})$ with $H\leq G$ and $h\in H$, whose corresponding species and primitive idempotent are denoted by $\sGnHh$ and $\eGnHh$, respectively. Through Section 4, we study fiber change maps, and state some results about their functoriality and naturality with respect to biset operations in a wide generality. Concretely, we prove that fiber change is always a morphism of inflation functors, and that it is a morphism of biset functors if and only if $f$ is injective on $\text{tor}(A)$. This last result will allow us to identify $\BnG$ as a subring of $D(G)$ commuting with all biset operations.

Finally, in Section 5 we focus on the conductors of $\BnG$. We prove that $c^{G,n}_{H,h}$ divides $[N_G(H,hH^{[n]}):H^{[n]}][H^{[n]}:H']_0$, where $[H^{[n]}:H']_0$ denotes the largest square-free divisor of $[H^{[n]}:H']$, and our main result states that these numbers are equal whenever $(n,\text{exp}(G)/n)=1$. Although the last condition might seem a bit too restrictive, it covers several cases of interest besides $B(G)$ and $D(G)$: if $k$ is a field of positive characteristic $p$ on which $X^{\text{exp}(G)}-1$ splits, the monomial $k$-representation ring $D^k(G)$ is naturally isomorphic to $\BnG$ if $\text{exp}(G)=p^an$ with $(p,n)=1$.

\subsection*{Notation}

Throughout this note $G$ is a finite group, and we write $\text{exp}(G)$ to denote its exponent. If $A$ is an abelian group, $\text{tor}(A)$ denotes its subgroup of torsion elements, and $C_n$ stands for a cyclic finite group of order $n$. If $g,x\in G$, we write $^gx$ and $x^g$ for $gxg^{-1}$ and $g^{-1}xg$, respectively, and $^gH$ and $H^{g}$ for $gHg^{-1}$ and $g^{-1}Hg$, respectively, if $H\leq G$; if $\phi:H\longrightarrow A$ is a homomorphism, we write $^g\phi:{^gH}\longrightarrow A$ for the conjugated map $^g\phi(x):=\phi(x^g)$, $x\in {^gH}$. The Burnside ring of $G$ is denoted by $B(G)$, and its monomial representation ring by $D(G)$. If $H\leq G$, we write $\varphi_H:B(G)\longrightarrow \Z$ for the $H$-th mark map sending the class of a $G$-set $X$ to $|X^H|$. Finally, if $n$ is a positive integer, we write $n_0$ for its square-free part, that is, the product of all the prime divisors of $n$.

\section{Abelianization and $n$-torsion}

Let $n$ be a positive integer. For an abelian group $A$, an element $a\in A$ is said to be an \textit{$n$-torsion element} if $a^n=e$. The set of $n$-torsion elements of $A$ is a subgroup that we denote by $A_n$, and we say that $A$ is an \textit{$n$-torsion group} if $A_n=A$. The category of $n$-torsion abelian groups is a full subcategory of the category of abelian groups, closed under quotients and products.

For a finite group $G$, we define $G^{[n]}$ to be a minimal normal subgroup such that the quotient $G/G^{[n]}$ is an $n$-torsion abelian group. Such a subgroup exists and it is unique and characteristic, and it can be constructed as $\langle G',\{g^n\;|\;g\in G\}\rangle$. The quotient $G/G^{[n]}$ together with the canonical projection $\pi:G\longrightarrow G/G^{[n]}$ satisfy the following universal property: for any group homomorphism $\phi:G\longrightarrow A$ into an $n$-torsion abelian group $A$, there is a unique group homomorphism $\overline{\phi}:G/G^{[n]}\longrightarrow A$ such that $\phi=\overline{\phi} \circ \pi$.

\begin{lemma} \thlabel{lemma1}
    Let $G$ be a finite group, $n,m\in \N$.
    \begin{enumerate}
        \item $G^{[exp(G)]}=G'$, $G^{[1]}=G$;
        \item if $n$ divides $m$, then $G^{[m]}\leq G^{[n]}$;
        \item if $H\leq G$, then $H^{[n]}\leq G^{[n]}$;
        \item $(G/G')^{[n]}=G^{[n]}/G'$;
        \item $G^{[n]}=G^{[n']}=G^{[n'']}$, where $n'=(n,\text{exp}(G))$ and $n''=(n,\text{exp}(G/G'))$.
    \end{enumerate}
\end{lemma}
\begin{proof}
    Assertions 1 to 4 are straightforward. To prove 5, note that the containments $G^{[n]}\leq G^{[n']}\leq G^{[n'']}$ are clear from part 2. Expressing $n''=an+b\text{exp}(G/G')$ for some $a,b\in \Z$, then for any $g\in G$ we have that
    $$g^{n''}=(g^a)^{n}(g^b)^{\text{exp}(G/G')}\equiv (g^a)^{n} \;\text{mod }G',$$
    hence $G^{[n'']}\leq G^{[n]}$.
\end{proof}

\begin{lemma}
    Let $G$ be a finite group and $n\in \N$. Then for any abelian group $A$, there are natural isomorphisms
    $$\text{Hom}(G,A)_n\cong \text{Hom}(G,A_n)\cong \text{Hom}(G/G^{[n]},A_n)\cong \text{Hom}(G/G^{[n]},A).$$
\end{lemma}
\begin{proof}
    The first isomorphism is clear as the operation in $\text{Hom}(G,A)$ is pointwise multiplication of functions, thus $\text{Im}(\phi)\leq A_n$ for all $\phi\in \text{Hom}(G,A)_n$. The second isomorphism follows by the universal property of $G/G^{[n]}$. The last isomorphism follows since the homomorphic image of an $n$-torsion group is also $n$-torsion.
\end{proof}

When $A=\Q/\Z$, the hom-set group $\text{Hom}(G,\Q/\Z)$ is the well-known \textit{Pontryagin dual} of $G$. The group $\Q/\Z$ is isomorphic to the group $\text{tor}(\C^{\times})$, for which $(\C^{\times})_n$ is cyclic of order $n$ generated by a primitive $n$-th root of $1$. Moreover, for any field $\K$, its group of torsion units $\text{tor}(\K^{\times})$ can be embedded in $\Q/\Z$ as a subgroup. This is clear if we consider the $p$-primary part of $\text{tor}(\K^{\times})$ for each prime $p$, which is either cyclic of order a power of $p$ or isomorphic to $\Z_{p^{\infty}}$, the $p$-primary part of $\Q/\Z$.

\begin{lemma}
    Let $G$ be a finite group and $n$ a positive integer. Then 
    \begin{enumerate}
        \item $Hom(G,C_n)$ and $G/G^{[n]}$ are isomorphic;
        \item $\text{Hom}(\text{Hom}(G,C_n),\C^{\times})$ and $G/G^{[n]}$ are naturally isomorphic;
        \item If $\K$ is a field and $n=|\K^{\times}_{\text{exp}(G/G')}|$, then $\text{Hom}(\text{Hom}(G,\K^{\times}),\C^{\times})$ and $G/G^{[n]}$ are naturally isomorphic.
    \end{enumerate}
\end{lemma}
\begin{proof}
    By the previous lemma, $\text{Hom}(G,C_n)$ and $\text{Hom}(G/G^{[n]},\C^{\times})$ are naturally isomorphic, and since $G/G^{[n]}$ is abelian, assertions 1 and 2 follow from the properties of the Pontryagin dual of a finite abelian group. Assertion 3 now follows since the image of any morphism $G\longrightarrow \K^{\times}$ lies in $\K^{\times}_{\text{exp}(G/G')}$, which is finite and cyclic.
\end{proof}

\section{Fibered Burnside rings}

This section is meant to be a summary of definitions and classic results related to fibered Burnside rings, and we refer the reader to \cite[\S 1]{BolCos} or \cite[\S 1]{Romero} for further details. We refer to \cite{Bouc} for a classic comprehensive text on biset functors.

Given a finite group $G$ and an abelian group $A$, a \textit{finite $A$-fibered $G$-set} is a $G\times A$-set with finitely many orbits on which the action of $A$ is free. The $A$-fibered $G$-sets together with the equivariant functions form a category $_G\text{set}^A$, provided with finite disjoint unions and a tensor product. The \textit{$A$-fibered Burnside ring} of $G$ is defined as the Grothendieck ring of the category $_G\text{set}^A$ with respect to disjoint union, denoted by $\BAG$. It is a free abelian group and a commutative ring with $1$ with a basis given by the isomorphism classes of transitive $A$-fibered $G$-sets, which are in bijection with orbits of monomial pairs: a \textit{monomial pair} or \textit{$A$-subcharacter} of $G$ is a pair $(H,\phi)$ with $H\leq G$ and $\phi\in \text{Hom}(H,A)$; the group $G$ acts by conjugation on the set $\MAG$ consisting of all monomial pairs, and we write $[H,\phi]_G$ for the orbit of $(H,\phi)$ and $N_G(H,\phi)$ for its stabilizer. The ring $\BAG$ is then isomorphic to the free abelian group on the set of orbits $G\backslash \MAG$, with multiplication given by
$$[H,\phi]_G[K,\psi]_G:=\sum_{g\in [H\backslash G/K]}[H\cap {^gK},\phi|_{H\cap {^gK}}{(^g\psi)}|_{H\cap {^gK}}]_G,$$
for $(H,\phi),(K,\psi)\in \MAG$, where $[H\backslash G/K]$ denotes a set of representatives of the double cosets $H\backslash G/K$, and the basis element $[G,1]_G$ acts as multiplicative identity. From now on we identify these rings as the same and we will refer to $G\backslash \MAG$ as the \textit{standard basis} or \textit{monomial basis} of $\BAG$. 

The group $G$ acts on the product ring $\prod_{H\leq G} \Z \text{Hom}(H,A)$ via conjugation maps, and the \textit{ghost ring} of $\BAG$ is defined as the subring $\gBAG:=\left(\prod_{H\leq G} \Z \text{Hom}(H,A)\right)^G$ of fixed points of this action. It comes with an injective ring homomorphism $\rho=\rho_G:\BAG \longrightarrow \gBAG$, known as the \textit{mark morphism}, mapping a basis element $[H,\phi]_G$ to the tuple $\rho([H,\phi]_G)$ whose $K$-th coordinate for $K\leq G$ is 
$$\sum_{\substack{gH\in G/H,\\
K\leq {^gH}}}{^g\phi|_K}=\sum_{\psi\in Hom(K,A)}\gamma^G_{(K,\psi),(H,\phi)}\psi \in \Z \text{Hom}(K,A),$$
where $\gamma^G_{(K,\psi),(H,\phi)}=\#\{gH\in G/H\;|\;K\leq {^gH},\psi={^g\phi}|_{K}\}$. The ghost ring is also free of the same rank as $\BAG$, and $\text{Coker}(\rho)$ is a torsion group annihilated by $|G|$.

\begin{exx}
    If $\K$ is a field, a \textit{finite $G$-line bundle over $\K$}, also known as a \textit{monomial $\K$-representation of $G$}, is a finite dimensional $\K G$-module $M$ together with a set $\mathcal{L}$ of $1$-dimensional subspaces that are permuted by the action of $G$ and such that $M=\bigoplus_{L\in \mathcal{L}}L$. Finite $G$-line bundles form a category $_{\K G}mon$ with direct sums and tensor products, and its Grothendieck ring $D^{\K}(G)$ (denoted by ${R^{ab}_{\K}}_+(G)$ in \cite{Bol2}), known as the \textit{monomial $\K$-representation ring of $G$}, is naturally isomorphic to $B^{\K^{\times}}(G)$. When $\K=\C$, we get the classic monomial representation ring. See \cite[\S 1, \S 5]{Bol2} for details.
\end{exx}

For a commutative ring $R$, we consider the scalar extension $B^A_R(G):=R\otimes \BAG$ as an $R$-algebra. We can easily identify $R\otimes \gBAG$ with the subring $\widetilde{B^A_R}(G)$ of fixed points of $\prod_{H\leq G}R\text{Hom}(H,A)$ for the action induced by conjugation maps just as before, and the extended mark morphism $\rho:B^A_R(G) \longrightarrow \widetilde{B^A_R}(G)$ is a map between torsion free $R$-algebras. When $R=\K$ is a field of characteristic zero, the map $\rho:B^A_{\K}(G)\longrightarrow \widetilde{B^A_{\K}}(G)$ is an isomorphism since $\text{Coker}(\rho)$ is a $\K$-vector space annihilated by $|G|$. Moreover, if $\BAG$ has finite rank, the $\K$-algebra $B^A_{\K}(G)$ is semisimple, and we say that $\K$ is a \textit{splitting field} for $\BAG$ if $B^A_{\K}(G)$ is isomorphic to an algebra of the form $\K^N$ for some positive integer $N$. Any extension of a splitting field of $\BAG$ is again a splitting field.

Just as other representation rings, fibered Burnside rings can be regarded as biset functors. More precisely, if $U$ is a finite $(H,G)$-biset and $X$ is an $A$-fibered $G$-set for finite groups $H$ and $G$, we can construct an $A$-fibered $H$-set $U\otimes_G X$ as the subset of $U\times_G X$ (regarding $X$ as a $(G,A)$-biset) consisting of all the elements on which $A$ acts freely. Passing to the fibered rings, we get a map of abelian groups $B^A(U):\BAG \longrightarrow \BAH$ for each $U$ that is additive and functorial with respect to composition of bisets, thus defining a linear functor $B^A:\mathcal{C}\longrightarrow \text{Ab}$, where $\mathcal{C}$ stands for the biset category. Moreover, $B^A(\text{res}_{\phi}):\BAH \longrightarrow \BAG$ is a ring homomorphism for each group homomorphism $\phi:G\longrightarrow H$, and the Frobenius identities hold for $B^A(\text{ind}_{\phi})$, hence $B^A$ is a Green biset functor.

Since transitive bisets generate $B(H,G)$, it is useful to have an expression for a transtive biset applied to an element of the standard basis. We recall that for a subgroup $E\leq H\times G$, there are subgroups 
$$k_1(E)=\left\{h\in H\;\middle|\;(h,1)\in E\right\}\unlhd p_1(E)=\left\{h\in H\;\middle|\;\exists g\in G\;\text{s.t.}\;(h,g)\in E\right\}\leq H,$$
and 
$$k_2(E)=\left\{g\in G\;\middle|\;(1,g)\in E\right\}\unlhd p_2(E)=\left\{g\in G\;\middle|\;\exists h\in H\;\text{s.t.}\;(h,g)\in E\right\}\leq G$$ 
satisfying 
$$p_1(E)/k_1(E)\cong E/\left(k_1(E)\times k_2(E)\right)\cong p_2(E)/k_2(E).$$
For another subgroup $T\leq K\times H$, the composite of $T$ and $E$ is defined as
$$T\star E=\left\{(k,g)\in K\times E\;\middle|\;\exists h\in H\;\text{s.t.}\;(k,h)\in T,(h,g)\in E\right\}\leq K\times G.$$
If $(K,\phi)\in \MAG$ and $\Gamma_{\phi}(K)\leq G\times A$ is the graph of $\phi$, then by the Mackey formula for transitive bisets we have
$$\left(H\times G/E\right)\times_G\left(G\times A/\Gamma_{\phi}(K)\right)\cong \bigsqcup_{g\in [p_2(E)\backslash G/K]} H\times A/E\star {^{(g,1)}\Gamma_{\phi}(K)},$$
as $H\times A$-sets, where the right-hand side expression is a decomposition in orbits, so the $A$-free part of it is precisely the union of those components for which $E\star {^{(g,1)}\Gamma_{\phi}(K)}$ is the graph of some subcharacter, which is the case if and only if $k_2(E\star {^{(g,1)}\Gamma_{\phi}(K)})=\{e\}$ if and only if $k_2(E)^g \cap K\leq \text{ker}(\phi)$. For such $g$, we set $L_g:=p_1(E\star {^g\Gamma_{\phi}(K)})\leq H$, and we define a morphism $\psi_g:L_g\longrightarrow  A$ as follows: if $y\in L_g$, there is some $x\in {^gK}$ such that $(y,x)\in E$, then $\psi_g(y):=\phi(x^g)$. Passing to the fibered rings, we get
\begin{align}
  B^A\left(H\times G/E\right)\left([K,\phi]_G\right) &= \sum_{\substack{g\in [p_2(E)\backslash G/K],\\
k_2(E)^g \cap K\leq \text{ker}(\phi)}} [L_g,\psi_g]_H.
\end{align}

Formulae for basic biset operations on basic elements are well known:
    \begin{enumerate}
        \item If $H\leq G$, then $\text{ind}^G_H\left([K,\phi]_H\right)=[K,\phi]_G$ for $(K,\phi)\in \MAH$, and 
        $$\text{res}^G_H\left([L,\psi]_G\right)=\sum_{g\in [H\backslash G/L]}[H\cap {^gL},{^g\psi}|_{H\cap {^gL}}]_H$$ 
        for $(L,\psi)\in \MAG$;
        \item If $N\unlhd G$ and $\pi:G\longrightarrow G/N$ is the canonical projection, then $\text{inf}^G_{G/N}\left([H/N,\phi]_{G/N}\right)=[H,\phi\circ \left(\pi|_H\right)]_G$ for $(H/N,\phi)\in \mathcal{M}^A_{G/N}$, and
        $$\text{def}^{G}_{G/N}\left([K,\psi]_G\right)=\begin{cases}
            [KN/N,\psi^{\pi}]_{G/N} &\text{if }K\cap N\leq \text{ker}(\psi),\\
            0 &\text{otherwise},
        \end{cases}$$
        for $(K,\psi)\in \MAG$, where $\psi^{\pi}(kN):=\psi(k)$ is well-defined when $K\cap N\leq \text{ker}(\psi)$.
        \item If $\alpha:G\longrightarrow H$ is an isomorphism, then $\text{iso}(\alpha)\left([K,\phi]_G\right)=[\alpha(K),\phi\circ (\alpha^{-1}|_{\alpha(K)})]_H$ for $(K,\phi)\in \MAG$. If $H\leq G$ and $g\in G$, then $^g[K,\phi]_H=[{^gK},{^g\phi}]_{^gH}$ for $(K,\phi)\in \MAH$.
    \end{enumerate}

\section{Species and idempotents}


A \textit{species} of $\BAG$ is a ring homomorphism from $\BAG$ to the field of complex numbers. We write $\text{Sp}(\BAG):=\text{Hom}_{ring}(\BAG,\C)$ for the set of species of $\BAG$. Each species $s$ of $\BAG$ extends to a unique $\C$-algebra homomorphism $s:B^A_{\C}(G)\longrightarrow \C$, giving an injection of $\text{Sp}(\BAG)$ into $\text{Hom}_{\C-alg}(B^A_{\C}(G),\C)$. Boltje \& Y{\i}lmaz have determined the species when $A_{\text{exp}(G)}$ is finite, which we recall here.

Assume that $A$ is an abelian group with finite $\text{exp}(G)$-torsion, and let $\K$ be a field of characteristic zero containing a root of $1$ of order $\text{exp}(A_{\text{exp}(G)})$. Let $\mathcal{D}^A_G$ be the set of all ordered pairs $(H,\Phi)$ for $H\leq G$ and $\Phi \in \text{Hom}(\text{Hom}(H,A),\K^{\times})$; it is a finite $G$-set under $G$-conjugation, and we write $[H,\Phi]_G$ for the orbit of a pair $(H,\Phi)$, and $N_G(H,\Phi)$ for its stabilizer. For each $(H,\Phi)$, set $s^{G,A}_{H,\Phi}$ to be the composite of maps
$$\xymatrix{
\BAG\ar[rr]^{\text{res}^G_H} &&\BAH\ar[rr]^-{\pi_H} &&\Z \text{Hom}(H,A)\ar[rr]^-{\Phi} &&\K
},$$
where $\pi_H$ maps $[K,\phi]_H$ to $\phi$ if $K=H$ and to $0$ otherwise, and the last arrow is the linear extension of $\Phi$. It is a ring homomorphism, and we have that $s^{G,A}_{H,\Phi}=s^{G,A}_{K,\Psi}$ if $(H,\Phi)$ and $(K,\Psi)$ are $G$-conjugates. For a basis element $[K,\phi]_G$, we have
$$s^{G,A}_{H,\Phi}\left([K,\phi]_G\right)=\sum_{\substack{
gK\in G/K,\\
H\leq {^gK}
}}\Phi\left((^g\phi|_H)^{-1}\right).$$

\begin{theorem}[Boltje - Y{\i}lmaz, \cite{BolYil} Theorem 3.1]
    The map
    $$s^{G,A}:B^A_{\K}(G)\longrightarrow \prod_{[H,\Phi]_G\in G\backslash \mathcal{D}^A_G}\K, \; x\mapsto \left(s^{G,A}_{H,\Phi}(x)\right)_{[H,\Phi]_G}$$
    is a $\mathbb{K}$-algebra isomorphism. In particular, every $\mathbb{K}$-algebra homomorphism from $B^A_{\K}(G)$ to $\K$ is of the form $s^{G,A}_{H,\Phi}$ for some $(H,\Phi)\in \mathcal{D}^A_G$, and $s^{G,A}_{H,\Phi}=s^{G,A}_{K,\Psi}$ if and only if $[H,\Phi]_G=[K,\Psi]_G$.
\end{theorem}

Taking $\K = \C$, we get that the species of $\BAG$ are exactly the maps $s^{G,A}_{H,\Phi}$ for $(H,\Phi)\in \mathcal{D}^A_G$, and they are in one-to-one correspondence with the orbits of $\mathcal{D}^A_G$; in fact, the species take values in $\Z[\zeta]$, where $\zeta$ is a primitive $n$-th root of $1$, $n=\text{exp}(A_{\text{exp}(G)})$. Since $B^{A}_{\K}(G)\cong \prod_{[H,\Phi]_G\in G\backslash \mathcal{D}^A_G}\K$, then for each $(H,\Phi)\in \mathcal{D}^A_G$ there is a primitive idempotent $e^{G,A}_{H,\Phi}$ in $B^{A}_{\K}(G)$ determined by the property
$$s^{G,A}_{K,\Psi}\left(e^{G,A}_{H,\Phi}\right)=\begin{cases}
    1 &\text{if }[K,\Psi]_G=[H,\Phi]_G,\\
    0 &\text{otherwise.}
\end{cases}$$

\begin{theorem}[Boltje - Y{\i}lmaz, \cite{BolYil} Thm. 3.2] \thlabel{BolYilIds}
    For each $(H,\Phi)\in \mathcal{D}^A_G$, we have that
    \begin{align}
        e^{G,A}_{H,\Phi}=\frac{1}{|N_G(H,\Phi)||\text{Hom}(H,A)|}\sum_{\substack{
        K\leq H,\\
        \Phi|_{K^{\perp}}=1
        }}\sum_{\phi\in \text{Hom}(H,A)}|K|\mu(K,H)\Phi(\phi^{-1})[K,\phi|_K]_G
    \end{align}
    in $B^A_{\K}(G)$, where $K^{\perp}:=\text{ker}\left(\text{res}^H_K\right)$ and $\mu$ is the Möbius function of the poset of subgroups of $G$.
\end{theorem}

\begin{remark}
    We can drop the condition $\text{ker}\left(\text{res}^G_H\right)\leq \text{ker}(\Phi)$ on the right-hand side of Equation 2, and still have
    \begin{align}
        e^{G,A}_{H,\Phi}=\frac{1}{|N_G(H,\Phi)||\text{Hom}(H,A)|}\sum_{\substack{
        K\leq H,\\
        \phi\in \text{Hom}(H,A)
        }
        }|K|\mu(K,H)\Phi(\phi^{-1})[K,\phi|_K]_G,
    \end{align}
    as we can see if we replace $\text{res}^H_Ke^{H}_{\Psi}$ by $|\text{Hom}(H,A)|^{-1}\sum_{\phi\in \text{Hom}(H,A)}\Phi(\phi^{-1})\phi|_K$ in Equation 9 of \cite[Theorem 3.2]{BolYil} using Equation 15 of the proof. Then Equation 3 is analogous to the formula for a primitive idempotent in $D_{\Q[\zeta]}(G)$ given by Müller in \cite{Muller}, but the additional condition on $K$ discards some null terms.
\end{remark}

We now focus on the case $A=C_n$, a finite cyclic group of order $n$. We can assume that $n$ is a divisor of $\text{exp}(G)$ and take $C_n$ to be the subgroup $(\C^{\times})_n$ of $\C^{\times}$, which will be useful later. Since $H/H^{[n]}$ and $\text{Hom}(\text{Hom}(H,C_n),\C^{\times})$ are naturally isomorphic for each $H\leq G$, where a coset $hH^{[n]}$ corresponds to the evaluation map $ev_h:\text{Hom}(H,C_n) \longrightarrow \C^{\times},\phi\mapsto \phi(h)$, we identify $\DnG$ with the set of pairs $(H,hH^{[n]})$, on which $G$ acts by conjugation too, and we write $\sGnHh$ instead of $s^{G,C_n}_{H,ev_h}$. In this setting, we have the following:

\begin{cor} \thlabel{Species}
There is a bijection
$$G\backslash \mathcal{D}^{C_n}_G \longleftrightarrow \text{Sp}(\BnG)$$
sending $[H,hH^{[n]}]_G$ to the species $\sGnHh$.
\end{cor}

Noting that $|\text{Hom}(H,C_n)|=[H:H^{[n]}]$, then from \thref{BolYilIds} we get:

\begin{cor} \thlabel{Idempotents}
    For each $(H,hH^{[n]})$ in $\DnG$, we have that
    \begin{align}
        e^{G,n}_{H,h}&=\frac{|H^{[n]}|}{|N_G(H,hH^{[n]})||H|}\sum_{\substack{
        K\leq H,\\
        \phi\in \text{Hom}(H,C_n)
        }}|K|\mu(K,H)\overline{\phi(h)}[K,\phi|_K]_G \\
        &=\frac{|H^{[n]}|}{|N_G(H,hH^{[n]})||H|}\sum_{\substack{
        K\leq H,\\
        \text{ker}(\text{res}^H_K)\leq \text{ker}(ev_h)
        }}\sum_{\phi\in \text{Hom}(H,C_n)}|K|\mu(K,H)\overline{\phi(h)}[K,\phi|_K]_G
    \end{align}
    in $B^{C_n}_{\Q[\zeta]}(G)$.
\end{cor}

The effect of biset operations on primitive idempotents was studied in \cite[\S 4]{BolYil}. For the purposes of this note, we will only need the following cases.

\begin{lemma} \thlabel{indresids}
    Let $H\leq G$ and $h\in H$. Then
    \begin{enumerate}
    \item $\text{res}^G_H\left(\eGnHh\right) = \sum_{\substack{uH^{[n]}\in H/H^{[n]},\\
        [H,uH^{[n]}]_G=[H,hH^{[n]}]_G
        }}e^{H,n}_{H,u}$;
        \item $\text{ind}_H^G\text{res}^G_H\left(\eGnHh\right)=[N_G(H):H]\eGnHh$;
        \item $\text{ind}_H^G\left(e^{H,n}_{H,h}\right)=[N_G(H,hH^{[n]}):H]\eGnHh$.
    \end{enumerate}
\end{lemma}
\begin{proof}
    Assertion 1 follows from part b of \cite[Proposition 4.1]{BolYil}, noting that $H$ acts trivially on $H/H^{[n]}$. The composite $\text{ind}^G_H\text{res}^G_H$ is the same as multiplication by $[H,1]_G$, which proves assertion 2. Finally, $\text{ind}^G_H\left(e^{H,n}_{H,u}\right)=\text{ind}^G_H\left(e^{H,n}_{H,h}\right)$ for all $uH^{[n]}$ in the $N_G(H)$-orbit of $hH^{[n]}$, thus
    $$[N_G(H):H]\eGnHh=[N_G(H):N_G(H,hH^{[n]})]\text{ind}^G_H\left(e^{H,n}_{H,h}\right),$$
    and assertion 3 now follows easily.
\end{proof}

\section{Fiber change maps}


If $f:A\longrightarrow C$ is a homomorphism of abelian groups, then $(H,f\circ \phi)\in\mathcal{M}^C_G$ for all $(H,\phi)\in \MAG$. Moreover, if $(K,\psi)={^g(H,\phi)}$ for some $g\in G$, then $f\circ \psi (k)=f(\phi(k^g))={^g(f\circ \phi)(k)}$ for all $k\in K$, proving that $(K,f\circ \psi)={^g(H,f\circ \phi)}$ in $\mathcal{M}^C_G$.

\begin{deff}
    The \textit{fiber change map induced by $f$} is the application $f_*:\BAG\longrightarrow B^C(G)$ sending $[H,\phi]_G$ to $[H,f\circ \phi]_G$. 
\end{deff}

Fiber change maps are in fact ring homomorphisms, as we prove in the following proposition.

\begin{prop}
    For every homomorphism $f:A\longrightarrow C$ of abelian groups, the fiber change map $f_*:\BAG \longrightarrow B^C(G)$ is a ring homomorphism. Moreover, fixing a finite group $G$, the assignments $A\mapsto \BAG$ and $f\mapsto f_*$ define a functor from the category of abelian groups to the category of unital commutative rings.
\end{prop}
\begin{proof}
    For all subcharacters $(H,\phi),(K,\psi) \in \MAG$, we have
\begin{align*}
    f_*([H,\phi]_G[K,\psi]_G) &=\sum_{g\in [H\backslash G/ K]}[H\cap {^gK},f\circ (\phi|_{H\cap {^gK}} {^g\psi|_{H\cap {^gK}}})]_G\\
    &=\sum_{g\in [H\backslash G/ K]}][H\cap {^gK},(f\circ \phi)|_{H\cap {^gK}} {^g(f\circ \psi)|_{H\cap {^gK}}}]_G\\
    &=f_*([H,\phi]_G)f_*([K,\psi]_G),
\end{align*}
and $f_*([G,1]_G)=[G,1]_G$. The assertion on the functoriality is straightforward.
\end{proof}

\begin{prop}\thlabel{InjectiveFiberChange}
    The fiber change map $f_*:\BAG \longrightarrow B^C(G)$ is injective if and only the restricted map $f:A_{\text{exp}(G)}\longrightarrow C_{\text{exp}(G)}$ is injective.
\end{prop}
\begin{proof}
    Suppose that $f$ is injective on $A_{\text{exp}(G)}$. If $f_*([H,\phi]_G)=f_*([K,\psi]_G)$ for some basis elements $[H,\phi]_G$ and $[K,\psi]_G$, then there is some $g\in G$ such that $f\circ \psi={^g(f\circ \phi)}=f\circ {^g\phi}$, hence $\psi={^g\phi}$ and $[H,\phi]_G=[K,\psi]_G$ in $\BAG$, which implies that $f_*$ is injective. Conversely, suppose that there is some $e\neq a\in \text{ker}(f)$ such that $a^{\text{exp}(G)}=e$. We can assume that $a$ is of prime order $p$ by replacing it for an adequate power. Then there is an element $e\neq g\in G$ of order $p$, and we define a morphism $\phi:\langle g\rangle \longrightarrow A$ sending $g$ to $a$. Then $[\langle g\rangle,\phi]_G\neq [\langle g\rangle,1]_G$ in $\BAG$, but $f_*([\langle g\rangle,\phi]_G)=[\langle g\rangle,1]_G=f_*([\langle g\rangle,1]_G)$ in $B^C(G)$, proving that $f_*$ is not injective.
\end{proof}

Recall that if $\text{exp}(G)=p^am$ for a prime $p$ with $a\geq 1$, and $p$ does not divide $m$, then there is an element of order $p^a$ in $G$. This is because $x^{p^a}=e$ for each $p$-element $x$, and $y^m=e$ for each $p'$-element $y$. If $p^b$ is the order of a $p$-element of maximum order, then $g^{p^bm}=e$ for all $g\in G$, hence $a=b$.

\begin{prop}\thlabel{IsoFiberChange}
    Let $f:A\longrightarrow C$ be a homomorphism of abelian groups which is injective on $A_{\text{exp}(G)}$. Then $f_*:\BAG\longrightarrow B^C(G)$ is an isomorphism if and only if $f:A_{\text{exp}(G)}\longrightarrow C_{\text{exp}(G)}$ is an isomorphism.
\end{prop}
\begin{proof}
    If $f|_{A_{\text{exp}(G)}}$ is not surjective onto $C_{\text{exp}(G)}$, then there is a prime divisor $p$ of $\text{exp}(G)$ and a $p$-element $e\neq c \in C_{\text{exp}(G)}$ of order $p^b\leq p^a=\text{exp}(G)_p$ that does not belong to $\text{Im}(f)$. Taking an element $e\neq g\in G$ of order $p^b$, we define a morphism $\phi:\langle g\rangle \longrightarrow C$, $g\mapsto c$. Then the element $[\langle g\rangle, \phi]_G$ of $B^C(G)$ does not belong to the image of $f_*$. The converse follows easily using the inverse of $f|_{A_{\text{exp}(G)}}$.
\end{proof}

In the following, we are going to determine to what extent the fiber change maps commute with biset operations. It is enough to consider the case of transitive bisets, so let $E\leq H\times G$ and $(K,\phi)\in \mathcal{M}_G^{A}$, then
\begin{align}
  f_*\left(B^A\left(H\times G/E\right)\left([K,\phi]_G\right) \right)&= \sum_{\substack{g\in [p_2(E)\backslash G/K],\\
k_2(E)^g \cap K\leq \text{ker}(\phi)}} [L_g,f\circ \psi_g]_H
\end{align}
and
\begin{align}
  B^C\left(H\times G/E\right)\left(f_*\left([K, \phi]_G\right)\right)&= \sum_{\substack{g\in [p_2(E)\backslash G/K],\\
k_2(E)^g \cap K\leq \text{ker}(f\circ \phi)}} [T_g,\theta_g]_H,
\end{align}
where $T_g=p_1(E\star {^{(g,1)}\Gamma_{f\circ \phi}(K)})\leq H$ and $\theta_g:T_g\longrightarrow C$ is defined as in Equation 1. We want to compare the right-hand sides of both equations. Fixing a set $[p_2(E)\backslash G/K]$, it is clear that if $g\in [p_2(E)\backslash G/K]$ is such that $k_2(E)^g\cap K\leq \text{ker}(\phi)$, then $k_2(E)^g\cap K\leq \text{ker}(f\circ \phi)$. For any $(U,\eta)\in \MAG$, we have that 
$$p_1(E\star \Gamma_{\eta}(U))=\left\{h\in H\;\middle|\;\exists u\in U\;\text{s.t.}\;(h,u)\in E\right\}$$ 
and so $p_1(E\star \Gamma_{\eta}(U))$ only depends on $E$ and $U$. Since $^{(g,1)}\Gamma_{\phi}(K)=\Gamma_{^g\phi}(^gK)$, we have $L_g=T_g$ for all $g$ such that $k_2(E)^g\cap K\leq \text{ker}(\phi)$, and moreover, if $x\in T_g=L_g$, then
$$\theta_g(x)=(f\circ \phi)(y^g)=f(\phi(y^g)))=f(\psi_g(x))$$ 
for some $y\in {^gK}$ such that $(x,y)\in E$, hence $\theta_g=f\circ \psi_g$. 

\begin{theorem} \thlabel{fiberchange1}
Let $f:A\longrightarrow C$ be a homomorphism of abelian groups. Then the fiber change maps $f_{*}:B^{A}(G) \longrightarrow B^{C}(G)$ for all finite groups define a morphism of biset functors if and only if $f$ is injective on $\text{tor}(A)$.
\end{theorem}
\begin{proof}
If $f$ is injective on $\text{tor}(A)$, then $\text{ker}(f\circ \phi)=\text{ker}(\phi)$ for all $(K,\phi)\in\MAG$, hence the right-hand sides of Equations 6 and 7 are equal for all $E\leq H\times G$ and all finite groups $H$ and $G$. Conversely, if $f$ is not injective on $\text{tor}(A)$, then there is a non-trivial element $a\in A$ of finite order $n$ such that $f(a)=1$. Let $C_n=\langle x\rangle$ be a cyclic group of order $n$ and take the morphism $\phi:C_n\longrightarrow A$ sending $x$ to $a$. Taking the deflation from $C_n$ to the trivial group, we get $\text{def}^{C_n}_1([C_n,\phi]_{C_n})= 0$, hence $f_*(\text{def}^{C_n}_1([C_n,\phi]_{C_n}))=0$ in $B^C(1)$. On the other hand, $f\circ\phi=1$ is the trivial morphism and so $f_*([C_n,\phi]_{C_n})=[C_n,1]_{C_n}\in B^C(C_n)$, then $\text{def}^{C_n}_1([C_n,1]_{C_n})=[1,1]_1$ in $B^C(1)$. Hence, $f_*\circ B^A(\text{def}^{C_n}_1)\neq B^C(\text{def}^{C_n}_1)\circ f_*$ on $B^A(C_n)$. 
\end{proof}

From the proof we see that when $f$ is not injective on $\text{tor}(A)$, there is a fiber change map that does not commute with a deflation. However, if we consider $B^A$ as an \textit{inflation functor}, that is, an additive functor over the subcategory $\mathcal{C}^{\triangleright}$ of $\mathcal{C}$ having all finite groups as objects and hom-sets generated by the classes of right-free bisets, then the fiber change maps always define a natural transformation (see \cite[Ex. 3.2.5]{Bouc}). 

\begin{theorem}
    Let $f:A\longrightarrow C$ be a morphism of abelian groups. Then the fiber change maps $f_*:\BAG\longrightarrow B^C(G)$ for all finite groups define a morphism of inflation functors.
\end{theorem}
\begin{proof}
     It is known that a transitive biset $H\times G/E$ is right-free if and only if $k_2(E)=\{e\}$, so let $E\leq H\times G$ be such that $k_2(E)=\{e\}$. Considering Equations 6 and 7, the condition $k_2(E)^g\cap K \leq \text{ker}(\phi)$ is trivially satisfied for all $(K,\phi)\in \MAG$ and all $g \in [p_2(E)\backslash G/K]$. Since $T_g=L_g$ and $\theta_g=f\circ \psi_g$ for all $g$, then the right-hand sides of both equations are equal.
\end{proof}

For each $H\leq G$, the pushout map $f_*:\text{Hom}(H,A)\longrightarrow \text{Hom}(H,C)$, $\phi\mapsto f\circ \phi$ induces a group homomorphism $\text{Hom}(\text{Hom}(H,C),\K^{\times})\longrightarrow \text{Hom}(\text{Hom}(H,A),\K^{\times})$, $\Psi\mapsto \Psi\circ f_*$ by precomposition. The collection of these maps now gives a map from $\mathcal{D}^C_G$ to $\mathcal{D}^A_G$ which is compatible with the $G$-actions.

\begin{prop}
    Let $f:A\longrightarrow C$ be a homormophism of abelian groups with finite $\text{exp}(G)$-torsion. Then $s^{G,C}_{H,\Psi}\circ f_*=s^{G,A}_{H,\Psi\circ f_*}$ for $(H,\Psi)\in \mathcal{D}^C_G$, and
        $$f_*\left(e^{G,A}_{H,\Phi}\right)=\sum_{\substack{
        [H,\Psi]_G\in G\backslash \mathcal{D}^C_G,\\
        [H,\Phi]_G= [H,\Psi\circ f_*]_G\in G\backslash\mathcal{D}^A_G
        }}e^{G,C}_{H,\Psi}$$
    for $(H,\Phi)\in \mathcal{D}^A_G$.
\end{prop}
\begin{proof}
   For $(L,\psi)\in \MAG$, we have
\begin{align*}
    s^{G,C}_{H,\Phi}\circ f_* \left([L,\psi]_G\right) & = s^{G,C}_{H,\Phi}\left([L,f\circ \psi]_G\right)\\
    & = \sum_{\substack{gL\in G/L,\\
    H\leq {^gL}}}\Phi\left({^g\left(f\circ \psi\right)|_H}\right)\\
    & = \sum_{\substack{gL\in G/L,\\
    H\leq {^gL}}}\left(\Phi\circ f_*\right)\left({^g \psi|_H}\right)\\
    & = s^{G,A}_{H,\Phi\circ f_*}\left([L,\psi]_G\right),
\end{align*}
and the last identity now follows from $s^{G,C}_{H,\Psi}\left(f_*\left(e^{G,A}_{H,\Phi}\right)\right)=s^{G,A}_{H,\Psi\circ f_*}\left(e^{G,A}_{H,\Phi}\right)$.
\end{proof}

For a divisor $n$ of $\text{exp}(G)$ and a divisor $t$ of $n$, let $i_{t,n}:C_t\hookrightarrow C_n$ be the natural inclusion as subgroups of $\C^{\times}$, and $\pi_{n,t}:C_n\longrightarrow C_t$, $\zeta\mapsto \zeta^{n/t}$. The map 
$$(i_{t,n})_*:B^{C_t}(G)\longrightarrow \BnG$$ 
is a quite simple injection, mapping $[H,\phi]_G\mapsto[H,\phi]_G$ for $(H,\phi)\in \mathcal{M}^{C_t}_G\subseteq \MnG$; on the other hand,
$$(\pi_{n,t})_*:\BnG\longrightarrow B^{C_t}(G)$$
sends $[K,\phi]_G$ to $[K,\phi^{n/t}]_G$ for $(K,\phi)\in \MnG$, and it is in general not surjective: if $G=C_{p^2}\times C_p=\langle x,y\rangle$ with $o(x)=p^2$ and $o(y)=p$, $n=p^2$ and $t=p$, then taking the map $\eta:G\longrightarrow C_p=\langle \zeta\rangle$ given by $x\mapsto \zeta$, $y\mapsto \zeta$, the element $[G,\eta]_G$ in $B^{C_p}(G)$ is not in the image of $(\pi_{p^2,p})_*$. However, if $(n/t,t)=1$, then $\zeta_t\mapsto \zeta_t^{n/t}$ is an automorphism of $C_t$, hence $(\pi_{n,t})_*$ is surjective. By the previous proposition,
$$\sGnHh\circ (i_{t,n})_*=s^{G,t}_{H,h},$$ 
and 
$$s^{G,t}_{H,h}\circ (\pi_{n,t})_*=s^{G,n}_{H,h^{n/t}}$$
for $H\leq G$ and $h\in H$.

\begin{cor}
    Let $t$ be a divisor of $n$. Then
    \begin{enumerate}
        \item $(i_{t,n})_*\left(e^{G,t}_{H,h}\right)=\sum_{\substack{
        [K,kK^{[n]}]_G\in G\backslash \mathcal{D}^{C_n}_G,\\
        [K,kK^{[t]}]_G=[H,hH^{[t]}]_G
        }}e^{G,n}_{K,k}$;
        \item $(\pi_{n,t})_*\left(\eGnHh\right)=\sum_{\substack{
        [K,kK^{[t]}]_G\in G\backslash \mathcal{D}^{C_t}_G,\\
        [K,k^{n/t}K^{[n]}]_G=[H,hH^{[n]}]_G
        }}e^{G,t}_{K,k}$;
        \item $(\pi_{n,1})_*\left(\eGnHh\right)=\begin{cases}
            e^{G,1}_{H,1} &\text{if } h\in H^{[n]},\\
            0 &\text{otherwise.}
        \end{cases}$
    \end{enumerate}
\end{cor}

\section{The conductors of primitive idempotents for cyclic fiber groups}

Let $A$ be an abelian group with finite ${\text{exp}(G)}$-torsion. Let $\K$ be a field of characteristic zero containing a root $\zeta$ of $1$ of order $\exp(A_{\text{exp}(G)})$, and let $\mathcal{O}\subseteq \K$ be its subring of algebraic integers. For $(H,\Phi)\in \mathcal{D}^A_G$, the \textit{conductor} of the primitive idempotent $e^{G,A}_{H,\Phi}$ over $\mathcal{O}$ is the minimum positive integer $t$ such that $te^{G,A}_{H,\Phi}$ lies in $B^A_{\mathcal{O}}(G)$, denoted by $c^{G,A}_{H,\Phi}$. Conductors exist and do not depend on $\K$ and $\mathcal{O}$ because the coefficients of $e^{G,A}_{H,\Phi}$ with respect to the monomial basis all lie in $\Q[\zeta]$ and $\mathcal{O}\cap \Q[\zeta]=\Z[\zeta]$. Since integers $t$ with the property that $te^{G,A}_{H,\Phi}$ belongs to $B^A_{\mathcal{O}}(G)$ form  an ideal in $\Z$, then any $t$ with that property must be divisible by $c^{G,A}_{H,\Phi}$.

In the following, we will consider $A=C_n$, and think of $\BnG$ as a subring of $D(G)$ for the natural inclusion, and write $\cGnHh$ for the conductor of $\eGnHh$. In this setting, we are going to prove the following assertion.

\begin{theorem} \thlabel{GRThm}
    Let $G$ be a finite group, and let $n$ be a positive divisor of $\text{exp}(G)$ such that $(n,\text{exp}(G)/n)=1$. Then for $(H,hH^{[n]})\in \DnG$, the conductor of the primitive idempotent $\eGnHh$ in $B^{C_n}_{\Q[\zeta]}(G)$ is
    \begin{align}
        c^{G,n}_{H,h}=[N_G(H,hH^{[n]}):H^{[n]}][H^{[n]}:H']_0.
    \end{align}
\end{theorem}

These formulae generalize those already known for the extreme cases $n=1$ and $n=\text{exp}(G)$ that correspond to the case of the Burnside ring and the monomial representation ring respectively:  for $H\leq G$,
$$c^{G,1}_{H,1}=[N_G(H,1H^{[1]})):H^{[1]}][H^{[1]}:H']_0=[N_G(H):H][H:H']_0$$ 
since $H^{[1]}=H$, as stated in \cite{DressVall}, \cite{HaIsOz} and \cite{Nicolson}, and 
$$c^{G,\text{exp}(G)}_{H,h}=[N_G(H,hH^{[\text{exp}(G)]}):H^{[\text{exp}(G)]}][H^{[\text{exp}(G)]}:H']_0=[N_G(H,hH'):H']$$ 
for $(H,hH')\in \mathcal{D}_G$ since $H^{[\text{exp}(H)]}=H'$, as proved by Müller in \cite[Theorem 3.5]{Muller}. We are going to prove \thref{GRThm} by showing that the left and right numbers in Equation 8 divide each other. To prove that $[N_G(H,hH^{[n]}):H^{[n]}][H^{[n]}:H']_0$ divides $c^{G,n}_{H,h}$, we will proceed as in Müller's proof. We will make use of Boltje's integrality criterion \cite[Theorem 2.2]{Bol1} which will allow us to decide whether an element of $\widetilde{B^A_R}(G)$ belongs to $\rho(B^A_R(G))$. We state it here for $A=C_n$ and $R=\Z[\zeta]$.

\begin{prop}[Boltje]\thlabel{BolCong}
    Let 
    $$x=\left(\sum_{\phi\in \text{Hom}(H,C_n)}a_{(H,\phi)}\phi\right)_{H\leq G}$$ 
    be an element in $\gBnGz$. Then $x\in \rho\left(\BnGz\right)$ if and only if the congruence
        \begin{align}
            \sum_{(H,\phi)\leq (K,\psi)\in \mathcal{M}^{C_n}_{N_G(H,\phi)}}\mu(H,K)a_{(K,\psi)}\; \equiv \; 0 \; mod\;[N_G(H,\phi):H]
        \end{align}
        holds for all $(H,\phi)\in \MnG$, where $\mu$ stands for the Möbius function of the poset of subgroups of $G$.
\end{prop}

We will also make use of the following generalized version of \cite[Lemma 3.3]{Muller}.

\begin{lemma}\thlabel{MullerLemma}
    Let $H\leq G$ and consider the map $res^G_H:\text{Hom}(G,C_n)\longrightarrow \text{Hom}(H,C_n)$, $\phi\mapsto \phi|_H$. Then
    \begin{enumerate}
        \item $\text{Im}(\text{res}^G_H)\cong HG^{[n]}/G^{[n]}$;
        \item $|(\text{res}^G_H)^{-1}(\psi)|=[G:HG^{[n]}]$ for $\psi\in \text{Im}(\text{res}^G_H)$;
        \item For $g\in G$,
        $$\sum_{\phi\in (\text{res}^G_H)^{-1}(\psi)}\phi(g)=\begin{cases}
            [G:HG^{[n]}]\psi(g) &\text{if }g\in HG^{[n]},\\
            0 &\text{else.}
        \end{cases}$$
    \end{enumerate}
\end{lemma}
\begin{proof}
    We have a commutative diagram
    $$\xymatrix{
    \text{Hom}(G,C_n)\ar@{<->}[d]\ar[rrrr]^{\text{res}^G_H} &&&&
    \text{Hom}(H,C_n)\ar@{<->}[d]\\
    \text{Hom}(G/G^{[n]},\C^{\times})\ar[rr]_{\text{res}^{G/G^{[n]}}_{HG^{[n]}/G^{[n]}}} &&\text{Hom}(HG^{[n]}/G^{[n]},\C^{\times})\ar@{^{(}->}[rr] &&\text{Hom}(H/H^{[n]},\C^{\times})
    }$$
    where the vertical arrows are the natural isomorphisms and the bottom-right arrow is the dual of the projection of $H/H^{[n]}$ onto $HG^{[n]}/G^{[n]}$. Hence, $\text{Im}\left(\text{res}^G_H\right)\cong \text{Im}\left(\text{res}^{G/G^{[n]}}_{HG^{[n]}/G^{[n]}}\right)$, and all of the assertions follow applying \cite[Lemma 3.3]{Muller} for $HG^{[n]}/G^{[n]}\leq G/G^{[n]}$.
\end{proof}

We recall a classic result by Hawkes, Isaacs and Özaydin that can be found in \cite{HaIsOz} as Theorem 4.5.

\begin{lemma}\thlabel{Rota}
    Let $H\leq G$. Then $[N_G(H):H]$ divides $\mu(H,G)[G:HG']_0$.
\end{lemma}

Finally, we will need a different expression for $\eGnGg$ to guarantee that the element $\rho\left([G:G^{[n]}]\eGnGg\right)$ of $\widetilde{B^{C_n}_{\Q[\zeta]}}(G)$ actually lies in $\gBnGz$. Although most of it are straightforward adaptations of the techniques in Müller's proof, there are some subtle differences as we work in a more general framework and we believe that it is worth it to go through the details. So let $e^G_G\in B_{\Q}(G)$ be the idempotent corresponding to the $G$-th mark $\varphi_G$, and consider an expression
$$e^G_G=\sum_{H\in [\mathcal{S}_G]}a_H[G/H]$$
in terms of the standard basis of $B_{\Q}(G)$, where $a_H$ is a rational number for each $H$. For every $\lambda\in \text{Hom}(G,C_n)$, set
$$x_{\lambda}:=\sum_{H\in [\mathcal{S}_G]}a_H[H,\lambda|_H]_G\in B^{C_n}_{\Q}(G).$$
Note that $x_1$ agrees with the image of $e^G_G$ under the map $[G/H]\mapsto [H,1]_G$. For $K\leq G$, we have 
$$\pi_K\text{res}^G_K([H,\lambda|_H]_G)=\sum_{\substack{
gH\in G/H,\\
K\leq {^gH}
}}{^g\lambda|_K}=\sum_{\substack{
gH\in G/H,\\
K\leq {^gH}
}}{\lambda|_K}=\varphi_K(G/H)\lambda|_K$$
where $^g\lambda=\lambda$ since $G$ acts trivially on $\text{Hom}(G,C_n)$. Therefore,
\begin{align}
\pi_K\text{res}^G_K(x_{\lambda})=\varphi_K(e^G_G)\lambda|_K=\begin{cases}
    \lambda, &\text{if }K=G,\\
    0, &\text{otherwise};
\end{cases}
\end{align}
in particular, the element $y_{\lambda}:=\rho(x_{\lambda})$ lies in $\widetilde{B^{C_n}}(G)$. Then applying $\rho$ to the sum $\sum_{\lambda\in \text{Hom}(G,C_n)}\overline{\lambda(g)}x_{\lambda}$ in $B^{C_n}_{\Q[\zeta]}(G)$, we have that
\begin{align}
    y_g:=\rho\left(\sum_{\lambda\in \text{Hom}(G,C_n)}\overline{\lambda(g)}x_{\lambda}\right)=\sum_{\lambda \in \text{Hom}(G,C_n)}\overline{\lambda(g)}y_{\lambda}
\end{align}
lies in $\widetilde{B^{C_n}_{\Z[\zeta]}}(G)$. Any species $\sGnHh$ with $H$ a proper subgroup vanishes on $\sum_{\lambda\in \text{Hom}(G,C_n)}\overline{\lambda(g)}x_{\lambda}$ since $\pi_H\text{res}^G_H$ does, while
$$s^{G,n}_{G,z}\left(\sum_{\lambda\in \text{Hom}(G,C_n)}\overline{\lambda(g)}x_{\lambda}\right)=\sum_{\lambda\in \text{Hom}(G,C_n)}\overline{\lambda(g)}\lambda(z)=\begin{cases}
    [G:G^{[n]}] &\text{if }zG^{[n]}=gG^{[n]},\\
    0 &\text{otherwise},
\end{cases}$$
by the column orthogonality relations for $G/G^{[n]}$. Hence, we have
\begin{align}
    e^{G,n}_{G,g}=\frac{1}{[G:G^{[n]}]}\sum_{\lambda\in \text{Hom}(G,C_n)}\overline{\lambda(g)}x_{\lambda}\in B^G_{\Q[\zeta]}(G),
\end{align}
and so $\rho\left([G:G^{[n]}]\eGnGg\right)=y_g\in \widetilde{B^{C_n}_{\Z[\zeta]}}(G)$.

\begin{prop} \thlabel{GRProp2}
    Let $(H,hH^{[n]})\in \DnG$. Then $[N_G(H,hH^{[n]}):H^{[n]}][H^{[n]}:H']_0\eGnHh$ belongs to $\BnGz$.
\end{prop}
\begin{proof}
We first consider $(H,hH^{[n]})=(G,gG^{[n]})$. From the previous discussion, we have that $y_g=\rho\left([G:G^{[n]}]e^{G,n}_{G,g}\right)\in \gBnGz$, so we now analyse the left-hand sum in Congruence 9 for $y_g=\left(\sum_{\psi\in \text{Hom}(K,C_n)}a_{(K,\psi)}\psi\right)_{K\leq G}$ and a monomial pair $(H,\phi)$. The congruence is trivially satisfied whenever $H$ is not normal since then $N_G(H,\phi)$ is a proper subgroup, and same when $\phi$ is not the restriction of some $\lambda:G\longrightarrow C_n$ since then $K$ has to be a proper subgroup for all $(H,\phi)\leq (K,\psi)$. So assume that $H\unlhd G$ and $\phi=\lambda|_H$ for some $\lambda:G\longrightarrow C_n$, then
\begin{align}
        \sum_{(H,\phi)\leq (K,\psi)\in \mathcal{M}^{C_n}_{N_G(H,\phi)}}\mu(H,K)a_{(K,\psi)}&=\mu(H,G)\sum_{
        \lambda\in \left(\text{res}^G_H\right)^{-1}(\phi)
        }\overline{\lambda(g)}\\
        &=\begin{cases}
            \mu(H,G)[G:HG^{[n]}]\overline{\phi(g)} &\text{if }g\in HG^{[n]},\\
            0 &\text{else}
        \end{cases}
\end{align}
by part 3 of \thref{MullerLemma}. In any case, because $[HG^{[n]}:HG']$ divides $[G^{[n]}:G']$, then $[G:HG']_0$ divides $[G:HG^{[n]}][G^{[n]}:G']_0$, and by \thref{Rota} we have that $[N_G(H):H]$ divides $\mu(H,G)[G:HG^{[n]}][G^{[n]}:G']_0$. By \thref{BolCong}, the element 
$$[G^{[n]}:G']_0y_g=\rho\left([G:G^{[n]}][G^{[n]}:G']_0\eGnGg\right)$$ 
belongs to $\rho\left(\BnGz\right)$, hence $[G:G^{[n]}][G^{[n]}:G']_0e^{G,n}_{G,g}$ belongs to $\BnGz$. 

For the general case, note that
\begin{align*}
    [N_G(H,hH^{[n]}):H^{[n]}][H^{[n]}:H']_0e^{G,n}_{H,n} &=[H:H^{[n]}][H^{[n]}:H']_0\text{ind}^G_H\left(e^{H,n}_{H,h}\right)\\
    &=\text{ind}_H^G\left([H:H^{[n]}][H^{[n]}:H']_0e^{H,n}_{H,h}\right).
\end{align*}
by \thref{indresids}, and since we have proved that $[H:H^{[n]}][H^{[n]}:H']_0e^{H,n}_{H,h}$ belongs to $B^{C_{n}}_{\Z[\zeta]}(H)$, then
$[N_G(H,hH^{[n]}):H^{[n]}][H^{[n]}:H']_0e^{G,n}_{H,n}$ belongs to $\BnGz$.
\end{proof}

We now prove a partial converse to \thref{GRProp2}.

\begin{prop}\thlabel{GRProp1}
    Let $n$ be a divisor of $\text{exp}(G)$ such that $(n,\text{exp}(G)/n)=1$, and $(H,hH^{[n]})\in \DnG$. Then $[N_G(H,hH^{[n]}):H^{[n]}][H^{[n]}:H']_0$ divides $c^{G,n}_{H,h}$.
\end{prop}
\begin{proof}
    Let $t$ be a non-zero integer such that $te^{G,n}_{H,h}\in \BnGz$. From Equation 4, we see that the coefficient of $[H,1]_G$ in $e^{G,n}_{H,h}$ with respect to the monomial basis is
    $$\frac{|H^{[n]}|}{|N_G(H,hH^{[n]})||H|}|H|\mu(H,H)1(h)=\frac{1}{[N_G(H,hH^{[n]}):H^{[n]}]},$$
    therefore, $[N_G(H,hH^{[n]}):H^{[n]}]$ divides $t$. Set $s=[N_G(H,hH^{[n]}):H^{[n]}]^{-1}t$, and let
    $$y:=[N_G(H,hH^{[n]}):H^{[n]}]e^{G,n}_{H,h}.$$
    
    Since $\text{exp}(H/H')$ divides $\text{exp}(G)$, then $H^{[n]}/H'=\left(H/H'\right)^{[n]}$ is a product of $p$-primary components of $H/H'$ for some prime divisors $p$ of $\text{exp}(G)/n$, hence $[H:H^{[n]}]$ and $[H^{[n]}:H']$ are relatively prime. If $p$ is a prime divisor of $[H^{[n]}:H']$, we consider the set
    $$\mathcal{S}_p:=\left\{K\leq H\;|\;[H:K]=p\right\},$$
    which is non-empty as $p$ is also a divisor of $[H:H']$. Those $K\in \mathcal{S}_p$ that are normal in $H$ are in bijection with the subgroups of $H/H'$ of index $p$, and so the number of these subgroups is congruent to $1$ modulo $p$. If $K\in \mathcal{S}_p$ is not normal in $H$, then its $H$-conjugacy class has size $p$ and is contained in $\mathcal{S}_p$. Putting all together, we have
    $$|\mathcal{S}_p| \equiv |\{K\unlhd H\;|\;[H:K]=p\}| \equiv 1\;\text{mod }p.$$
    The set $\mathcal{S}_p$ splits into equivalence classes for the relation of being $G$-conjugate, and we write $[K]$ for the class of $K$. For every $K\in \mathcal{S}_p$, we have $|(\text{res}^H_K)^{-1}(1)|=[H:KH^{[n]}]=1$, since it is a divisor of both $[H:K]=p$ and $[H:H^{[n]}]$ which are coprime. Therefore, $\text{res}^H_K$ is injective, and the coefficient of $[K,1]_G$ in $y$ with respect to the monomial basis is
    $$\frac{1}{|H|}\left(\sum_{\substack{
    L\leq H,\\
    L =_{G}K
    }}\sum_{\substack{
    \phi\in \text{Hom}(H,C_n),\\
    \phi|_L=1
    }}|L|\mu(L,H)\overline{\phi(h)}\right) 
    =\frac{-1}{[H:K]}\left(\sum_{L\in [K]}1(h)\right)
    =-\frac{\#[K]}{p}$$
    since $\mu(L,H)=-1$ and $|L|=|K|$ for all $L\in [K]$. Because $|\mathcal{S}_p|$ is not divisible by $p$, there is a subgroup $K\in \mathcal{S}_p$ such that $\#[K]$ is not divisible by $p$. Then $sy\in \BnGz$ implies that $-s\#[K]/p$ is an integer, therefore that $p$ divides $s$. Hence $[H^{[n]}:H']_0$ divides $s$.
\end{proof}

\begin{proof}[Proof of \thref{GRThm}]
    The equality now follows from \thref{GRProp2} and  \thref{GRProp1}.
\end{proof}

\begin{exx}
    Let $\K$ be a field in which $X^{\text{exp}(G)}-1$ splits. If $\text{char}(\K)=p>0$ divides $|G|$ and $\text{exp}(G)=p^mn$ with $(p,n)=1$, then $D^{\K}(G)\cong \BnG$. For each $(H,hH^{[n]})\in \DnG$, we have two possible cases: if $p$ does not divide $[H:H']$, then $H^{[n]}=H'$, and $\cGnHh=[N_G(H,hH'):H']$; if $[H:H']=p^at$ for some $a\geq 1$ and $(p,t)=1$, then $[H:H^{[n]}]=t$ and $[H^{[n]}:H']=p^a$, hence $\cGnHh = [N_G(H,hH^{[n]}):H]tp$.
\end{exx}

\begin{remark}
    For any divisor $n$ of $\text{exp}(G)$, we always have that $[N_G(H,hH^{[n]}):H^{[n]}]$ divides $\cGnHh$, and by \thref{GRProp1}, $\cGnHh$ divides $[N_G(H,hH^{[n]}):H^{[n]}][H^{[n]}:H']_0$, hence, 
    $$r^{G,n}_{H,h}:=\frac{\cGnHh}{[N_G(H,hH^{[n]}):H^{[n]}]}$$ 
    is a divisor of $[H^{[n]}:H']_0$. If $(H,hH^{[n]})=(G,gG^{[n]})$, the number $r^{G,n}_{G,g}$ is the least positive integer such that $r^{G,n}_{G,g}\rho\left([G:G^{[n]}]\eGnGg\right)$ belongs to $\rho\left(B^{C_n}_{\Z[\zeta]}(G)\right)$, and by Boltje's integrality criterion and Equation 14, it is the minimum positive integer such that $\mu(H,G)[G:HG^{[n]}]r^{G,n}_{G,g}$ is divisible by $[N_G(H):H]$ for each $H\unlhd G$ such that $g\in HG^{[n]}$. Back to the general case, using $\text{ind}^G_H$ we have that a possibly better upper bound for $r^{G,n}_{H,h}$ is
    $$u^{G,n}_{H,h}=\text{lcm}\left\{\frac{[KH^{[n]}:KH']_0}{([H:KH^{[n]}],[KH^{[n]}:KH'])_0}\;\middle| \;K\unlhd H, h\in KH^{[n]}\right\}$$
    since we might be discarding some unnecessary primes appearing in $[H^{[n]}:H']_0$. Of course, the expression for $u^{G,n}_{H,h}$ is far from easy to compute or handle, but it could be the case, as in \thref{GRThm}, that it does not depend on $h$. In fact, if $(n,\text{exp}(G)/n)=1$, we have $([H:KH^{[n]}],[KH^{[n]}:KH'])=1$ for all $K$, and so $u^{G,n}_{H,h}=[H^{[n]}:H']_0=r^{G,n}_{H,h}$.
\end{remark}

\subsection*{Acknowledgments}

The authors thank the reviewer for their insightful suggestions. The first author was supported by CONAHCYT with an EPM Postdoctoral grant.

\end{document}